\documentclass[11pt]{article}
\usepackage[numbers]{natbib}
\usepackage{graphicx,amsmath,epsfig,amssymb,amsfonts,amsthm,enumerate,verbatim,todonotes,hyperref,mathtools, dsfont}

\usepackage{tikz,tikz-cd}
\usetikzlibrary{matrix,arrows}

\newtheorem{theorem}{Theorem}[section]
\newtheorem{corollary}[theorem]{Corollary}
\newtheorem{lemma}[theorem]{Lemma}
\newtheorem{propo}{Proposition}[section]
\newtheorem{conjecture}[theorem]{Conjecture}
\newtheorem{question}[theorem]{Question}

\newtheorem{notat}{Notation}[section]

\newtheorem{fact}{Fact}[section]

\newtheorem{note}{Note}
\theoremstyle{definition}
\newtheorem{definition}[theorem]{Definition}

\newcommand{\restrict}{\mathord{\upharpoonright}}
\newcommand{\concat}{\mathbin{\raisebox{1ex}{\scalebox{.7}{$\frown$}}}}

\begin{document}

\title{Ken's colorful questions}
\author{Iv\'an Ongay-Valverde\\\emph{Host Institution: Department of Mathematics and Statistics}\\\emph{York University, Toronto, ON, CA}\\\emph{Current Institution: Department of Mathematics}\\\emph{University of Toronto, Toronto, ON, CA}\\\emph{Email: ivan.ongay.valverde@gmail.com}}
\date{}
\maketitle
\begin{flushright}
\textit{Dedicated to Ken, Anne and their family.
}

\end{flushright}

\begin{abstract} The paper surveys some questions concerning \emph{coloring axioms} which grew out of the discussions the author had with his PhD advisor Ken Kunen.

\end{abstract}

\section{Introduction to coloring axioms}

In studying a question of his long time colleague, M. E. Rudin,
who asked whether \textbf{MA} and the failure of \textbf{CH} implies that every locally connected, hereditarily Lindel\"{o}f, compact space is metrizable Ken Kunen became interested in an example of Filippov \cite{MR0256350}.
Filippov had used a Luzin set to construct a locally connected, hereditarily Lindel\"{o}f, compact space that is not metrizable, and Filippov's
space is also hereditarily separable. Since \textbf{MA} + ¬\textbf{CH} implies that there are no Luzin sets, Kunen wondered whether
\textbf{MA} and the failure of \textbf{CH} might refute the existence of such a space. In \cite{KunenLocallyCompacta} he
discovered some interesting facets of \textbf{SOCA} and used these to show
that this is not the case.

Filippov's construction relies on the geometry of spheres and Ken was able to see in this a useful weakening of a Luzin set.
For $T\subseteq \mathbb R^n$ let $$T^* = \{x- y : x, y \in T \And x\neq y\}$$
and Kunen defines $T\subseteq \mathbb R^n$ to be {\em skinny}
if the closure of $\{x/\|x\| : x\in T^*\}$ is not the entire surface of the sphere in $\mathbb R^n$.
He calls a set $ E \subseteq \mathbb R^n$ {\em weakly Luzin} if $E$ is uncountable but every skinny subset of $E$
is countable. He then shows that the Filippov space constructed from a set $E$ has no uncountable discrete subsets if and only if $E$ is weakly Luzin. Furthermore, he shows that weakly Luzin sets and the entangled sets that play a prominent role in the study of colouring axioms have a nice common combinatorial generalization.

With these ideas he is able to show from \textbf{SOCA} that if $X$ is compact and $Y$ is compact metric
with $\pi: X \to Y$ continuous and if there is some uncountable $E \subseteq Y$
such that for all $y \in E$, there are three points $\{x_{i,y}\}_{i\in 3} \subseteq \pi^{-1}\{y\}$
with disjoint open neighbourhoods $U_{i,y}$ of $x_{i,y}$ with pairwise disjoint ranges under $\pi$
then X has an uncountable discrete subset. These results stimulated my interest in modifications of colouring axioms.

Following Kunen's \cite{ KunenLocallyCompacta, kunen2014set, hart2011arcs}, given a topological space $E$ we denote by $E^{\dagger}=E^{2}\setminus \{(x,x):x\in E\} $ the square of the space without the diagonal. A set $W\subseteq E^{\dagger}$ is \emph{symmetric} if $(x,y)\in W$ whenever $(y,x)\in W$. We refer to a symmetric $W\subseteq E^{\dagger}$ as a coloring. We shall consider various topological properties of the coloring, for example, if $W$ is open we say that it is an open coloring.
If there is $T\subseteq E$ such that $T^{\dagger}\subseteq W$ then we say that it is \emph{$W$-connected, $W$-homogeneous}, or \emph{open-homogeneous} (when $W$ is not clopen). If there is $T\subseteq W$ such that $T^{\dagger}\cap W=\emptyset$ then we say that it is \emph{$W$-free}, \emph{$W^{c}$-homogeneous} or \emph{closed-homogeneous} (when not clopen).

The fundamental concept dealt with in the paper is the \emph{Semi Open Coloring Axiom}:

\begin{definition}
Given a collection $\mathcal{X}$ of separable metric (or topological) spaces \textbf{SOCA($\mathcal{X}$)} is the statement ``Given an uncountable space $E\in \mathcal{X}$ and an open coloring $W\subseteq E^{\dagger}$, there is an uncountable $T\subseteq E$ which is either $W$-connected or $W$-free (in other words, $T$ is homogeneous)". \textbf{SOCA} is \textbf{SOCA($\mathcal{X}$)} when $\mathcal{X}$ is the class of all separable metric spaces.
\end{definition}

\textbf{SOCA} was arguably Ken's favourite colouring axiom;
In his book \cite{kunen2014set}, he worked out a complete proof of the consistency of \textbf{SOCA} and as mentioned above he put it to good use in \cite{hart2011arcs} and \cite{KunenLocallyCompacta}.

The axiom itself was introduced by Abraham, Rubin and Shelah in \cite{abraham1985consistency} together with another principle called there the \emph{Open Coloring Axiom} (\textbf{OCA}), which later is usually referred to as \textbf{OCA}${}_{[ARS]}$ to distinguish it from the ``other" open coloring axiom introduced by Todor\v{c}evi\'{c} \cite{todorcevic1989partition} sometimes denoted by \textbf{OCA}${}_{[T]}$ and lately called the \emph{Open Graph Axiom} (\textbf{OGA}) by Todor\v{c}evi\'{c} himself.

Both axioms are consequences of the \emph{Proper Forcing Axiom} (\textbf{PFA}). Given that the consistency of \textbf{PFA} requires large cardinals the following two somewhat vague questions seem natural:

\begin{question}\label{PFA colors}
How much of the strength of \textbf{PFA} can be expressed with coloring axioms?
\end{question}

\begin{question}
How how much of \textbf{PFA} is equiconsistent with \textbf{ZFC}? Can that be expressed with a coloring axiom?
\end{question}

The first question has been studied extensively by Todor\v{c}evi\'{c} and his school. In particular, he showed \cite{todorcevic1989partition} that \emph{Martin's Axiom} (\textbf{MA}) is equivalent to the statement that given a separable metric space $E$ of size less than $\mathfrak c$ every ccc\footnote{A coloring $W$ \emph{is ccc} if given $p_{\alpha}\in[E]^{<\omega}$, $\alpha<\omega_1$, such that $(p_{\alpha})^{\dagger}\subseteq W$ for all $\alpha<\omega_1$, there exist $\alpha, \beta<\omega_1$ such that $(p_{\alpha}\cup_{\beta})^{\dagger}\subseteq W$.}
coloring of $E^{\dagger}$ has an uncountable homogeneous set (see also \cite{todorcevic-velickovic}). Furthermore, Moore \cite{MooreOCAaleph2} showed that \textbf{OCA$_{[ARS]}$} and \textbf{OGA} together imply that $\mathfrak{c}=\aleph_2$.

In \cite{abraham1985consistency}, the relative consistency of \textbf{ZFC}+\textbf{OCA$_{[ARS]}$} + \textbf{SOCA} + ``\emph{There is an increasing set\footnote{We say that $A\subseteq \mathds{R}$ is an \emph{increasing set} if and only it is uncountable and for any one-to-one function $f\subseteq A^2$ between two disjoint uncountable subset of $A$ and any $n\in\omega$, there are $a_0<a_1<...<a_{n-1}$ such that $f(a_{i})< f(a_{i+1})$ for all $i< n-1$. }}" is established, while an increasing set is a counterexample to \textbf{OGA} as shown in \cite{todorcevic1989partition}. Hence, \textbf{SOCA} is weaker than \textbf{OGA}, and \textbf{OCA$_{[ARS]}$} does not imply \textbf{OGA}. Abraham-Rudin-Shelah also showed (in \cite{abraham1985consistency}) that \textbf{SOCA} does not imply \textbf{OCA$_{[ARS]}$}. So the only possible implications left are:

\begin{question}\label{OCA go SOCA}
Does \textbf{OCA$_{[ARS]}$} imply \textbf{SOCA}?
\end{question}

\begin{question}\label{OGA go OCA}
Does \textbf{OGA} imply \textbf{OCA$_{[ARS]}$}?
\end{question}

The most famous open problem about the coloring axioms is, of course:

\begin{question}[\cite{todorcevic1989partition}] Is \textbf{OGA} consistent with $\mathfrak c>\aleph_2$?
\end{question}

A positive answer to this question provides a negative answer to the previous one by the aforementioned result of Moore \cite{MooreOCAaleph2}. Moreover, the analogous question for
\textbf{OCA$_{[ARS]}$} has recently been solved in the affirmative by Gilton and  Neeman \cite{gilton2019abraham}.

We conclude this quick overview of coloring axioms by recalling another open problem raised by Todor\v{c}evi\'c in \cite{todorcevic1989partition} (see also \cite{BobanApplicationsOCA}):

\begin{question}[\cite{todorcevic1989partition}] Is \textbf{OGA}($\mathcal X$) consistent for the class $\mathcal X$ of all regular spaces without uncountable discrete subsets?
\end{question}

An analogous question can be also asked for \textbf{SOCA}:
\begin{question}
What is the largest family of topological spaces $\mathcal{X}$ such that \textbf{SOCA}($\mathcal{X}$) is relatively consistent with \textbf{ZFC}?
\end{question}

\section{Weakenings of SOCA}\label{Section SOCA}

In this section we consider two natural weaker versions of \textbf{SOCA}:

\begin{itemize}
\item $\textbf{Clopen SOCA}$ - the same as \textbf{SOCA} but only for clopen colorings.
\item $\textbf{Dense SOCA}$ - the same as \textbf{SOCA} but only for open dense colorings.
\end{itemize}

To understand the relation between these principles it is useful to consider the following concept:

\begin{definition}
Given a topological property $P$, an uncountable separable metric space $E$ and a coloring $W$ over $E$ we say that \emph{$W$ can be reduced to a $P$-coloring} if there is an uncountable set $T\subseteq E$ such that $W\cap T^{\dagger}$ is a coloring with property $P$ in the topology of $T^{\dagger}$ induced from $E^{\dagger}$.
\end{definition}

Obviously, given any collection $\mathcal{X}$ of separable metric spaces, the axiom \textbf{SOCA($\mathcal{X}$)} is stronger than both \textbf{Clopen SOCA($\mathcal{X}$)}, \textbf{Dense SOCA($\mathcal{X}$)}.

\begin{propo}\label{sons of soca} Assume \textbf{SOCA}. Then
\begin{enumerate}
\item (\textbf{Clopen Reduction}) All open colorings on a separable metric space can be reduced to clopen colorings.
\item (\textbf{Dense Reduction}) Every open coloring on a separable metric space can be reduced to an open dense or empty coloring.
\end{enumerate}
\end{propo}

\begin{proof} We shall prove both items simultaneously. Given a space $E$ and an open coloring $W$ over it, using \textbf{SOCA} there is an uncountable $T\subseteq E$ that is homogeneous. If $T$ is $W$-connected, then $T^{\dagger}\cap W=T^{\dagger}$ which is clopen and open dense in $T^{\dagger}$. Otherwise, $T$ is $W$-free so $T^{\dagger}\cap W=\emptyset$ which is clopen and empty.
\end{proof}

Recall that an uncountable set $A\subseteq \mathds{R}$ is \emph{$2$-entangled} if for every uncountable collection
of pairwise disjoint $2$-element subsets of $A$, there are pairs $(x_{1}, y_{1})$, $(x_{2}, y_{2})$, $(w_{1}, z_{1})$, $(w_{2}, z_{2})$ in the collection such that $x_{1}<x_{2}$ and $y_{1}<y_{2}$, $w_{1}<w_{2}$ but $z_{1}>z_{2}$. The existence of a $2$-entangled set follows e.g. from \textbf{CH}, while in \cite{abraham1985consistency} it is shown that
\textbf{SOCA} implies that there are no $2$-entangled sets.
Interestingly, \textbf{Clopen SOCA} suffices and, consequently, \textbf{Clopen SOCA} is not a theorem of \textbf{ZFC}.

\begin{theorem} Assuming \textbf{Clopen SOCA} or \textbf{Dense Reduction} there are no $2$-entangled sets.
\end{theorem}

\begin{proof} It suffices to show that given any uncountable linearly ordered separable metric space $X$, the \emph{increasing} open coloring
$$W=\{((a,b), (c,d)) : \ a<c \leftrightarrow b<d \}$$
of $X^{2}$ can be reduced to a clopen coloring.

To see this, recursively construct $T=\{x_{\xi}=(x_{1}^{\xi}, x^{\xi}_{2}): \xi<\omega_1\}$ as follows: Assume $x_{\xi}=(x_{1}^{\xi}, x^{\xi}_{2})$ for $\xi<\alpha$ have already been chosen. Since $\alpha$ is countable $X^{2}\setminus\left(\bigcup_{\xi<\alpha}(X\times\{x_{2}^{\xi}\}\cup \{x_{1}^{\xi}\}\times X)\right)\neq \emptyset$. Take $x_{\alpha}$ to be any point in it.

Now, given two distinct $(x,y), (z,w)\in T$, we have that $x\neq y$ and $z\neq w$, hence, $((x,y), (z,w))$ is either in $W$ or in $\{((a,b), (c,d)): a<c \leftrightarrow b>d \}$, but both are open sets. Therefore, $T\cap W$ is clopen in $T$.

\medskip

We start the proof from the \textbf{Dense Reduction} by noting the following:

\begin{fact} \label{clopen to open dense}
The only way to reduce a clopen coloring to an open dense one is to have an uncountable homogeneous open set.
\end{fact}

To see this, let $W$ be a clopen coloring over $E$ and assume that $T\subseteq E$ is an uncountable set that reduces $W$ to open dense. Notice that, given $(a,b)\in T^{\dagger}\cap W^{c}$, as $T$ reduces $W$ to open dense, $T^{\dagger}\cap W \cap U\neq \emptyset$ for every open set $(a,b)\in U$. On the other hand, as $W$ is clopen in $E$, there is an open set $V$ such that $(a,b)\in V\subseteq W^{c}$. Clearly, these two assumptions are contradicting each other. Therefore, $T^{\dagger}\cap W^{c}=\emptyset$ making $T$ open-homogeneous. This concludes the proof of the fact.

Let $W$ be the increasing coloring on $A^{2}$ where $A$ is a $2$-entangled set. We know that $W$ can be reduced to a clopen coloring, so lets assume that it is clopen. Notice that for any $T\subseteq A^{2}$, $T^{\dagger}\cap W$ is clopen. Using our fact, the only way to get an open dense coloring or an empty coloring would be finding an uncountable homogeneous set, which does not exist when $A$ is a $2$-entangled set. Then the existance of a $2$-entangled set implies the failure of \textbf{Dense Reduction}.
\end{proof}

The increasing coloring can be used to show that there is an open coloring which is not reducible to an open dense coloring, hence the phrase ``or empty" in the \textbf{Dense Reduction} is necessary. To see this take a subset of $\mathds{R}^{2}$ consisting of the graph of a decreasing function and the graph of a countable partial increasing function, and consider the increasing coloring
defined above.

Next we shall see that the axioms considered here imply that $\mathfrak b>\omega_1$ following Todor\v{c}evi\'c \cite{todorcevic1989partition} and Moore \cite{MooreContColor}.
We shall call an $\leq^*$-increasing unbounded chain in $\omega^{\omega}$ of minimal length a \emph{$\mathfrak{b}$-scale}.

\begin{propo}\label{Reduce clopen and scales} Each of the axioms
\textbf{Clopen SOCA}, \textbf{Dense SOCA} and \textbf{Clopen Reduction} implies $\mathfrak{b}>\aleph_1$.
\end{propo}\label{b scales not clopen}

\begin{proof} The fact that \textbf{Clopen SOCA} implies $\mathfrak{b}>\aleph_1$ was proved in \cite{MooreContColor}.

\smallskip

Now, assume, towards a contradiction, that every open coloring can be reduced to a clopen coloring and $\mathfrak{b}=\aleph_1$.
Take a $\mathfrak{b}$-scale and let $W$ be an open coloring such that \[W^{c}=\{(f,g): f\geq g \mbox{ or } f\leq g\}\]
where $f\leq g$ means that for all $n\in \omega$, $f(n)\leq g(n)$.
The above coloring has no uncountable homogeneous set when $\mathfrak{b}=\aleph_{1}$ (see \cite{todorcevic1989partition}). Furthermore, using the theory of oscillation of Todor\v{c}evi\'c \cite{todorcevic1989partition}, in any cofinal (for $\mathfrak{b}=\aleph_{1}$, uncountable) set of a $\mathfrak{b}$-scale there are $f$ and $h$, such that $h\leq^{\ast} f$ but $h\not\leq f$, so there is a value $m$ where $h(m)>f(m)$, so $W^{c}$ is closed but not clopen.
Now, since every uncountable subset of a $\mathfrak{b}$-scale is also a $\mathfrak{b}$-scale (when $\mathfrak{b}=\aleph_{1}$), $W$ can never be reduced to a clopen coloring. This contradicts our assumptions.

\smallskip
Finally, towards a contradiction, assume \textbf{Dense SOCA} and $\mathfrak{b}=\aleph_1$. Using a $\mathfrak{b}$-scale $B$ of size $\aleph_1$, we can eliminate countably many points to have that every open set is uncountable. Using the same coloring $W$ as above, every open set in $B^{\dagger}$ has a subset of the form
\[\{(f,h): \exists n,m\in \omega f(n)< h(n) \mbox{ and } f(m)>h(m)\},\]
which is an open set. So $W^{c}$ is nowhere dense (nwd). Furthermore, by a theorem of Todor\v{c}evi\'c \cite{todorcevic1989partition}, every $\mathfrak{b}$-scale has two functions, $f$ and $g$ such that $f<g$. Then, this coloring is open dense but it has no uncountable homogeneous set.
\end{proof}

In particular, none of the four axioms considered are theorems of \textbf{ZFC}.
It turns out that various combinations of these weakenings of \textbf{SOCA} recover the whole strength of \textbf{SOCA}.

\begin{theorem}
The following are equivalent:
\begin{enumerate}
\item \textbf{SOCA}
\item \textbf{Clopen SOCA} + \textbf{Clopen Reduction}
\item \textbf{Dense SOCA} + \textbf{Dense Reduction}
\item \textbf{Clopen Reduction} + \textbf{Dense Reduction}
\item \textbf{Dense SOCA} + \textbf{Clopen SOCA}
\end{enumerate}
\end{theorem}

\begin{proof}
We already know that \textbf{SOCA} implies all the other statements. The equivalence of Clauses 2 and 3 with 1 have identical natural proofs: first reducing any open coloring to one of the special kind and then apply the corresponding weakening of \textbf{SOCA} to that special coloring. The rest follows directly from the following two observations:

\begin{fact}\label{Open dense reduction implies clopen soca}
\textbf{Dense reduction} implies \textbf{Clopen SOCA}.
\end{fact}

\begin{fact}\label{Open dense soca implies clopen reduction}
\textbf{Dense SOCA} implies \textbf{Clopen reduction}.
\end{fact}

Fact \ref{Open dense reduction implies clopen soca} is a direct consequence of Fact \ref{clopen to open dense}. To show Fact \ref{Open dense soca implies clopen reduction}, let $W$ be an open coloring over $E$. Notice that $W\cup \mbox{int}(W^{c})$ is an open dense coloring over $E$, since its complement is $\partial W$.\footnote{$\partial W$ denotes the boundary of $W$. It is nowhere dense and it can be define either as $\overline{W}\cap\overline{W^{c}}$ or as $\overline{W}\setminus \mbox{int}(W)$}

Using \textbf{Dense SOCA} over $W\cup \mbox{int}(W^{c})$ we either get a closed-homogeneous uncountable set $T$, in which case, $T^{\dagger}\cap (W\cup \mbox{int}(W^{c}))=\emptyset$, so $T^{\dagger}\cap W=\emptyset$ is clopen, or we get an uncountable set such that $T^{\dagger}\subseteq W\cup \mbox{int}(W^{c})$. Because both $W$ and $W^{c}$ are open in $T^{\dagger}$ we have that $T$ reduces $W$ to a clopen coloring. This concludes the proof of the facts.
\end{proof}

We close this section with a few more natural questions:

\begin{question}
Is \textbf{Dense SOCA} equivalent to \textbf{Clopen Reduction}?
\end{question}

\begin{question}
Is \textbf{Clopen SOCA} equivalent to \textbf{Dense Reduction}?.
\end{question}

\begin{question}
Is \textbf{Dense SOCA} weaker than \textbf{SOCA}?
\end{question}

\begin{question}\label{Question Clopen equal SOCA}
Is \textbf{Clopen SOCA} weaker than \textbf{SOCA}?
\end{question}

Some of these would be settled by a positive answer to the following:

\begin{conjecture}
Every open coloring can be reduced to a clopen or to an open dense coloring. \end{conjecture}

As mentioned in the previous section, it is not known if \textbf{OCA$_{[ARS]}$} implies \textbf{SOCA}. If it turns out that the answer is negative, then the fact that \textbf{\mbox{OCA}}$_{[ARS]}$ implies \textbf{Clopen SOCA} (a clopen coloring is also a cover of $E^{\dagger}$ by open sets) settles Question \ref{Question Clopen equal SOCA} in the positive.

\section{Colorings and Baire spaces}\label{Section Baire}

Galvin \cite{galvin1968partition} (see also \cite{hart2011arcs}) showed that \textbf{SOCA} is true for every Polish space. In particular, in uncountable Polish spaces all open colorings can be reduced to clopen ones. Hence also any space that contains an uncountable Polish space has the same reduction property.
These \textbf{ZFC} results suggest the following:

\begin{question}
What is the largest family of topological spaces $\mathcal{X}$ such that \textbf{ZFC} implies \textbf{SOCA($\mathcal{X}$)}?
\end{question}

In a related work, Ramos-Garc\'{\i}a and Corona-Garc\'{\i}a \cite{ArietOCA} study the class of topological spaces for which \textbf{OGA} follows from \textbf{ZFC}.

\smallskip

Recall that a metric space $E$ is a \emph{Baire} space if no non-empty open subsets of $E$ is \emph{meager}, i.e. is not a union of countably many nowhere dense sets. Polish spaces are Baire by the Baire Category Theorem. Nevertheless, not all Baire spaces are Polish or contain a Polish space. For example \emph{Luzin} sets or \emph {Generalized Luzin} sets (\cite{kunen2014set}) are Baire spaces with no Polish subspace (these sets exists, for example, under \textbf{CH} or \textbf{MA}+$\neg$ \textbf{CH}, respectively).

Furthermore, as shown in Kunen's book \cite{kunen2014set}, it is possible to have a Luzin set as counterexample to \textbf{SOCA}, so it is consistent that not all Baire metric spaces satisfy \textbf{SOCA}. We will show here that, if \textbf{Clopen SOCA} is true for Baire spaces, then \textbf{SOCA} is also true for them. In order to prove this, the following notation will be useful:

\begin{notat}
Given $A\subseteq E^2$ and $e\in E$ let
\[A_e=\left\{y\in E \ : \ (y,e)\in A\right\}.\]
\end{notat}

\begin{theorem}\label{Baire clopen} Every open coloring over a separable metric Baire space can be reduced to a clopen one.
\end{theorem}

\begin{proof}
Let $E$ be a separable metric Baire space and let $W\subseteq E^{\dagger}$ be open and symmetric.

By definition, $\partial W$ is meager (nowhere dense) and, by the Kuratowski-Ulam Theorem, we have that the set
\[T_{0}=\{e\in E \ :\ (\partial W)_{e}\mbox{ is not nowhere dense}\}\]
is meager.

We can recursively construct a sequence $\langle x_{\xi} : \xi\in \omega_{1}\rangle$ contained in $E\setminus T_{0}$ such that, given $\alpha<\beta$, $x_{\beta}\notin (\partial W)_{x_{\alpha}}$. Once we construct this sequence we will be done: by symmetry of $W$, $(e,e')\in \partial W$ if and only if $(e',e)\in \partial W$. So, the above sequence will have the property that, given $\alpha$ and $\beta$, $(x_{\alpha}, x_{\beta})\notin \partial W$. In other words, there is an open set $O$ of $E^{\dagger}$ such that $ (x_{\alpha}, x_{\beta})\in O$ and either $O\cap W=\emptyset$ or $O\subseteq W$. Therefore, $\{x_{\xi}:\xi\in \omega_{1}\}^{\dagger}\cap W$ will be clopen in $\{x_{\xi}:\xi\in \omega_{1}\}^{\dagger}$.

For the construction, assume that we already selected $x_{\xi}$ for $\xi<\alpha$, for $\alpha<\omega_{1}$. Since $\alpha$ is countable, $E$ is a Baire space and $x_{\xi}\notin T_{0}$ we have that $E\setminus \left( \bigcup_{\xi<\alpha}(\partial W)_{x_{\xi}}\cup T_{0} \right)\neq \emptyset$. We just let $x_{\alpha}$ be an element of $E\setminus \left( \bigcup_{\xi<\alpha}(\partial W)_{x_{\xi}}\cup T_{0} \right)$.\end{proof}

\begin{corollary}
\textbf{Dense SOCA(Baire)} is true in \textbf{ZFC}.
\end{corollary}

\begin{proof}
If in the proof above we assume that $W$ is open dense, then $W^{c}=\partial W$. So, the proof above generates a $W$-connected set.
\end{proof}

\begin{corollary}
\textbf{SOCA(Baire)} and \textbf{Clopen SOCA(Baire)} are equivalent.
\end{corollary}

The following is a variation on a classical proof of Todor\v{c}evi\'c in \cite{todorcevic1989partition}.

\begin{theorem}\label{Luzin entangled}
Assuming \textbf{CH} there is a 2-entangled Luzin set.
\end{theorem}

\begin{proof}
Let $\{f_{\alpha}: \alpha<\mathfrak{c}\}$ be an enumeration of all continuous functions from $G_{\delta}$ subsets of $\mathds{R}$ to $\mathds{R}$ and let $\mathds{R}=\{x_{\alpha} :\alpha<\mathfrak{c}\}$ be an enumeration of the reals. Finally, let $\{N_{\alpha}: \alpha<\mathfrak{c}\}$ be an enumeration of all closed nwd sets.

We will construct the $2$-entangled Luzin set $\mathcal{S}\subseteq \mathds{R}$ by recursion. Assume that we already have $\mathcal{S}_{\alpha}=\{x_{\gamma_{\xi}}: \xi<\alpha\}$. Let
\[\gamma_{\alpha}=\min\{\beta : x_{\beta}\notin (\bigcup_{\xi<\alpha}N_{\xi})\cup B_{\alpha}=\emptyset\}\]
where \[B_{\alpha}=\{f_{\chi}(x_{\gamma_{\xi}}): \xi, \chi<\alpha\}.\]
Notice that $B_{\alpha}$ is meager (it is countable) and $(\bigcup_{\xi<\alpha}N_{\xi}\cup B_{\alpha})$ is a countable union of nwd (by \textbf{CH}), so they do not cover $\mathds{R}$. This shows that $\mathds{R}\setminus (\bigcup_{\xi<\alpha}N_{\xi}\cup B_{\alpha})$ is non empty and $\gamma_{\alpha}$ is well defined.

It is clear from the construction, and from the assumption that $\mathfrak{c}=\aleph_1$, that $\mathcal{S}$ is a Luzin set. To show that $\mathcal{S}$ is $2$-entangled we  follow the proof of Lemma 4.2 by Todor\v{c}evi\'c in \cite{todorcevic1989partition}. 
Let $\{(x_{\alpha_{\xi}}, x_{\beta_{\xi}}): \xi<\mathfrak{c}\}\subseteq \mathds{R}^{2}$ be a collection of size continuum of disjoint $2$-tuples of $\mathcal{S}$.

Let \[K=\{x_{\alpha_{\xi}}: \alpha_{\xi}< \beta_{\xi}<\mathfrak{c}\}.\] We can assume that this set is of size continuum. If not, we can run the argument interchanging the roles of $\alpha_{\xi}$ and $\beta_{\xi}$.

Now, we can define the function $g:K\rightarrow \mathds{R}$ such that \[x_{\alpha_{\xi}}\mapsto x_{\beta_{\xi}}.\]

Furthermore, define the set \[K_{0}=\{s\in K:|\omega_{g}(s)|\geq 2\},\] where\footnote{Here $B_{\frac{1}{n}}^{K}(s)$ is the ball of radius $\frac{1}{n}$ with center $s$ in $K$, a subset of $\mathds{R}$. } $\omega_{g}(s)=\bigcap_{n\in \omega}\overline{g[B_{\frac{1}{n}}^{K}(s)]}$ is the oscillation of $g$ at $s$, i.e., all the accumulation points of the images (under $g$) of sequences that converge to $s$. Notice that $|\omega_{g}(s)|=1$ if and only if $g$ is continuous at $s$.

Recall that any partial continouos function from $\mathds{R}^{n}$ to $\mathds{R}$ can be extended to a partial function whose domain is a $G_{\delta}$ set. With this and our construction of $\mathcal{S}$ we have that the set $K_{0}$ is of size continuum.

Given $s\in K_{0}$, let $a_{s}$, $b_{s}$ be two distinct elements in $\omega_{g}(s)$. Without loss of generality, we can assume that $a_{s}<b_{s}$. Let $r\in \mathds{Q}$ such that $a_{s}<r<b_{s}$. Since we only have countably many rational numbers, we may shrink $K_0$ in such a way that for all $s\in K_{0}$ the rational number $r$ is the same. Notice that $K_{0}$ still has size continuum. 

Take $t, s\in K_{0}$ such that $t< s$ and take disjoint intervals $I_{t}, I_{s}$ such that $t\in I_{t}$ and $s\in I_{s}$. By the definition of $a_{t}$, $a_{s}$, $b_{t}$ and $b_{s}$ there are $t_{0},t_{1}\in K\cap I_{t}$ and $s_{0}, s_{1}\in I_{s}\cap K$ such that $g(t_{0}), g(s_{0})<r<g(t_{1}),g(s_{1})$. Then for the pairs $(t_{0},g(t_{0}))$, $(s_{1}, g(s_{1}))$ we have $t_{0}<s_{1}$ and $g(t_{0})<g(s_{1})$; and for the pair $(t_{1}, g(t_{1}))$, $(s_{0}, g(s_{0}))$ we have $t_{1}<s_{0}$ but $g(t_{1})>g(s_{0})$.
\end{proof}

\begin{note}\label{entangled note}
In the proof of Theorem \ref{Luzin entangled}, we index the sequence with $\mathfrak{c}$ since the entangled part is true in \textbf{ZFC}. The use of \textbf{CH} is only to ensure that the set is Luzin.
\end{note}

As every Luzin set is Baire, we have the following:

\begin{corollary}
In a model where \textbf{CH} is true, \textbf{SOCA(Baire)} is false.
\end{corollary}

With this we can conclude that \textbf{Dense SOCA(Baire)} is weaker than \textbf{Clopen SOCA(Baire)} and \textbf{SOCA(Baire)}.

\section{SOCA for larger uncountable sets}\label{Section big SOCA}

In this section we shall consider higher cardinal extensions of \textbf{SOCA}. This is motivated indirectly by the recent attempts to ``lift" Baumgartner's Theorem \cite{baumgartner1973alphall} to $\aleph_2$, i.e. to prove the consistency of \emph{every two $\aleph_2$-dense sets of reals are order isomorphic} (see \cite{moorebaumgartner, gilton2017side}), and directly by the work of Shelah, Abraham and Rudin's \cite{abraham1985consistency} where they ask if it is possible to have a version of \textbf{SOCA} for $\aleph_{2}$. It fits in the general program to investigate the behaviour of
the continuum when it is bigger than $\aleph_{2}$ (as discussed in \cite{todorvcevic1997comparing}).

A straightforward generalization of \textbf{SOCA} is to ask that the set $T$ is as big as the space or bigger than a certain cardinal.

\begin{definition}
\textbf{SOCA($\kappa$)} is the statement: \emph{For all separable metric spaces $E$ of size bigger or equal to $\kappa$ and all open symmetric subsets $W$ of $E^{\dagger}$ there exist $T\subseteq E$ such that $|T|\geq \kappa$ and either $T^{\dagger}\subset W$ (open-homogeneous) or $T^{\dagger}\cap W=\emptyset$ (closed-homogeneous).}
\end{definition}

For some spaces, this axiom can be derived from \textbf{SOCA} and \textbf{MA}.

\begin{theorem}[\textbf{SOCA} and \textbf{MA}]\label{SOCA from SOCA}
Given a separable metric space $E$ of size $\kappa< \mathfrak{c}$, with $cof(\kappa)\neq \aleph_{0}$, and $W\subseteq E^{\dagger}$ open and symmetric such that all closed-homogeneous sets are countable, there is an open-homogeneous set $T\subseteq E$ of size $\kappa$.
\end{theorem}

\begin{proof}
Let 
\[\mathds{C}_{E,W}=\{p\in [E]^{<\omega} : \forall x, y\in p \ (x\neq y\rightarrow (x,y)\in W)\}.\]
It is enough to show that this is a ccc poset. Once we have this, since it is of size $\kappa<\mathfrak{c}$, \textbf{MA} will imply that it is the union of countably many filters. Notice that the union of each filter is an open-homogeneous set and, since $\{e\}\in \mathds{C}_{E,W}$ for all $e\in E$, we have that $E$ is a countable union of open-homogeneous sets so, it has an open-homogeneous set of size $\kappa$ (using its cofinality).

In order to show that $\mathds{C}_{E,W}$ is ccc take $\{p_{\alpha} \ : \ \alpha<\omega_{1}\}\subseteq \mathds{C}_{E,W}$. We will show that there exist $\alpha\neq \beta$ such that $p_{\alpha}$ is compatible with $p_{\beta}$.

Since $|p_{\alpha}|\in \omega$ for all $\alpha$ we may assume that $\{p_{\alpha} \ : \ \alpha<\omega_{1}\}$ is a $\Delta$-system with root $r$. Furthermore, we may assume that $|p_{\alpha}\setminus r|=|p_{\beta}\setminus r|=m$. So, we can write $p_{\alpha}\setminus r=\{x^{0}_{\alpha},..., x^{m-1}_{\alpha}\}$

Let $\mathcal{B}$ be a countable basis of $E$. For each $\alpha$ we can choose $V_{i}\in \mathcal{B}$ such that for all $i\in m$, $x_{i}^{\alpha}\in V_{i}$ and for all $i\neq j$, $V_{i}\times V_{j}\subseteq W$ (this is possible since $W$ is open and for all $i\neq j$ and all $\gamma$ $(x^{\gamma}_{i}, x^{\gamma}_{j})\in W$).

We will prove by induction on $m$ that there exists uncountable many elements that are compatible.

For $m=1$, let $T=\left\{ x^{0}_{\alpha} : \alpha<\omega_{1}\right\}$. We know that $T$ do not have any uncountable closed-homogeneous set, so we can use \textbf{SOCA($\aleph_{1}$)} to choose an uncountable open-homogeneous set. Let $S\subseteq \omega_{1}$ be the indexes of the elements of this open-homogeneous set. Notice that all the elements of $\{p_{\alpha} \ : \ \alpha\in S\}$ are compatible.

Assume we have the result for $m$, we will prove it for $m+1$.

First, using the induction hypothesis, take $\{q_{\alpha} \ : \ \alpha<\omega_{1}\}$ such that $\{q_{\alpha}\setminus \{x_{\alpha}^{m}\} \ : \ \alpha<\omega_{1}\}$ are compatible. Now, let $T=\{x^{m}_{\alpha} : \alpha<\omega_{1}\}$. We know that $T$ do not have any uncountable closed-homogeneous set, so we can use \textbf{SOCA($\aleph_{1}$)} to choose an uncountable open-homogeneous set. Let $S\subseteq \omega_{1}$ be the indexes of the elements $x_{\alpha}^{m}$ in that uncountable open-homogeneous set. Notice that all the elements of $\{q_{\alpha} \ : \ \alpha \in S\}$ are compatible: first, all the elements of $\{q_{\alpha}\setminus \{x_{\alpha}^{m}\} \ : \ \alpha\in S\}$ are compatible; also, we have that $(x_{\alpha}^{m}, x^{m}_{\beta})\in W$ for all $\alpha, \beta\in S$ and, finally, we have that $(x^{m}_{\alpha}, x^{j}_{\beta})\in V_{m}\times V_{j}\subseteq W$ for all $\alpha, \beta \in \omega_{1}$ and all $j\in m$.
\end{proof}

The assumption that $\kappa<\mathfrak{c}$ cannot be improved. Following Note \ref{entangled note} after Theorem \ref{Luzin entangled}, we can modify Lemma 4.2 of \cite{todorcevic1989partition} to show:

\begin{theorem}\label{2-entangled}
There is a set of size continuum $X$ such that for every collection of size continuum of disjoint $2$-tuples of $X$ there are $(x_{1}, y_{1})$, $(x_{2}, y_{2})$, $(w_{1}, z_{1})$, $(w_{2}, z_{2})$ such that $x_{1}<x_{2}$ and $y_{1}<y_{2}$, $w_{1}<z_{1}$ but $w_{2}>z_{2}$.
\end{theorem}

Theorem \ref{2-entangled} shows that the increasing coloring for $X^{2}$ cannot have a size continuum homogeneous set, so \textbf{SOCA($\kappa$)} can only be valid for $\kappa<\mathfrak{c}$.

\medskip

In order to prove \textbf{SOCA}, the autors of \cite{abraham1985consistency} used sequences of elementary submodels and assumed that their spaces only had countable closed-homogeneous sets using a combinatorial principle (\emph{there always exists fast clubs for (special) families of size continuum}). Following these ideas Moore and Todor\v cevi\' c in \cite{moorebaumgartner} introduced the principle \textbf{(**)} which we shall call \textbf{MTA}:

\begin{definition}[Moore, Todor\v cevi\' c \cite{moorebaumgartner}]
The \emph{Moore-Todor\v cevi\' c Axiom} for ($\omega_{2}$,$\aleph_{2}$, $\aleph_{1}$) (\textbf{MTA} or \textbf{MTA($\omega_{2},\aleph_{2}, \aleph_{1}$)}) is the statement:

\smallskip

\emph{If $\mathcal{F}$ is a collection of one-to-one functions from $\omega_{2}$ to $\omega_{2}$ and $|\mathcal{F}|\leq \aleph_{2}$ then there is $g:\omega_{2}\rightarrow \omega_{2}$ which is one-to-one such that
\begin{itemize}
\item for every $f\in \mathcal{F}$, $\{\alpha : f(\alpha)=g(\alpha)\}$ is countable (or of size $<\aleph_{1}$).
\item for every $f\in \mathcal{F}$, there is a countable (or of size $<\aleph_{1}$) set $D\subseteq \omega_{2}$ such that if $\alpha\neq \beta\in \omega_{2}\setminus D$ then $f(g(\alpha))\neq g(\beta)$.
\end{itemize}}

\end{definition}

In the same paper were \textbf{MTA} is introduced, the following equivalence is proved:

\begin{lemma}\label{Equiv MTA}
For each $\beta\in \omega_{2}\setminus \omega_{1}$, fix  a bijection $b_{\beta}:\beta\rightarrow \omega_{1}$. \textbf{MTA} is equivalent to the following statement: whenever $\mathcal{F}$ is a collection of at most $\aleph_{2}$ many one-to-one functions from $\omega_{1}$ to $\omega_{1}$, there is a one-to-one $g:\omega_{2}\rightarrow \omega_{2}$ such that, whenever $\beta$ is closed under $g$ and $f\in \mathcal{F}$, there is a countable $D\subseteq \beta$ such that:

\begin{itemize}
\item if $\xi\in \beta\setminus D$ then \[f(\min(b_{\beta}(\xi), b_{\beta}(g(\xi))))<\max(b_{\beta}(\xi), b_{\beta}(g(\xi))),\]
\item if $\eta\neq \xi\in \beta\setminus D$ then \[f(\min(b_{\beta}(g(\eta)), b_{\beta}(g(\xi))))<\max(b_{\beta}(g(\eta)), b_{\beta}(g(\xi))).\]
\end{itemize}
\end{lemma}

Notice that, using the \textbf{Axiom of Choice}, if the functions are countable-to-one, you can also construct a $g$ with the above mentioned properties. The following lemma was also discovered, independently, by Thomas Gilton \cite{giltonpersonal}.

\begin{lemma}\label{main Lemma continuum 2}
$(2^{\aleph_{0}}=\aleph_{2}+\mbox{ \textbf{MTA}})$ Given a separable metric space $E$ of size $\aleph_{2}$ and an open and symmetric $W\subseteq E^{\dagger}$ such that all closed-homogeneous sets are countable there is an $E_{0}\subseteq E$ of size $\aleph_{2}$ such that the finite open-homogeneous subsets of $E_{0}$ ordered by reverse inclusion form a ccc forcing.
\end{lemma}

\begin{proof}
The proof for this Lemma is inspired, indeed, by Todor\v{c}evi\'{c} and Moore \cite{MooreOCAaleph2} and by the proof of the consistency of \textbf{SOCA} with $\mathfrak{c}>\aleph_2$ in Abraham-Rudin-Shelah \cite{abraham1985consistency}. The inspiration of the former will be seen on the use of \textbf{MTA}. On the other hand, the strategy that Abraham-Rudin-Shelah used to prove their result was that given a metric space, they created a ``tower of models'' for each closed set of the metric space and its cartesian powers. After that, they ``combined'' all of them using their Axiom \textbf{A1} (which, essentially, guarantees a fast club) and, finally, use the club method.

We will have a similar approach (even using the club method at the end). Nevertheless, we will change \textbf{A1} for \textbf{MTA} and create ``towers of models'' for all the pairs of closed sets and an uncountable ordinal of $\omega_2$. We do this as follows:

Let $E=\{e_{\alpha} : \alpha<\omega_{2}\}$ and $\mathcal{B}=\{B_{n}: n\in \omega\}$ be a basis for its topology. Fix $\theta\geq \aleph_{3}$, and $b_{\beta}$ bijections between $\beta$ and $\omega_{1}$ for each $\beta\in \omega_{2}\setminus \omega_{1}$.

For each closed subset of $E^{n}$, say $F$, and each $\beta\in \omega_{2}\setminus \omega_{1}$ we define $\{M_{\gamma}^{F,\beta}: \gamma<\omega_{1}\}$ a continuous $\in$-chain of countable elementary
submodels of $\langle H(\theta), E, \mathcal{B}\rangle$ such that given $\alpha<\beta$ we ensure that $e_{\alpha}\in M_{b_{\beta}(\alpha)+1}^{F,\beta}$. Whenever $F$ and $\beta$ are clear from the context we will use the notation $M_{\gamma}$ for $\gamma\in \omega_1$.

Notice that, since there are only $\aleph_{2}$ pairs $(F, \beta)$, we just define $\aleph_2$ ``towers of models'' each one of them with $\omega_1$ models such that for all $\gamma\in \omega_1$, $F\in M^{F,\beta}_{\gamma}$ and \[\{e_{\alpha}: \alpha<\beta\}\subseteq \bigcup_{\delta\in \omega_1}M^{F,\beta}_{\delta}.\]

Once again, for $F$ a closed subset of $E^{n}$ and $\beta\in \omega_{2}\setminus \omega_{1}$, we define $f_{F, \beta}:\omega_{1}\rightarrow \omega_{1}$ such that:
\[f_{F,\beta}(\alpha)=\min\left\{\gamma\in\omega_{1} :\forall \xi<\alpha \left( e_{b_{\beta}^{-1}(\xi)}\in M_{\gamma}\right), \forall \delta \geq \gamma\left( e_{b_{\beta}^{-1}(\delta)}\notin M_{\alpha+1}\right)\right\}\]
using the sequence of elementary submodels corresponding to $F$ and $\beta$. These functions are well define, since each $M_{\chi}$ is countable, and each of of them is a countable-to-one function.

The combinatorial principle \textbf{MTA} ensures the existance of a one-to-one function $g:\omega_{2}\rightarrow\omega_{2}$ satisfying the conditions of Lemma \ref{Equiv MTA} for the set $\{f_{F, \beta}: \beta\in\omega_{2}\setminus\omega_{1}, n\in \omega, F\subseteq E^{n}\}$.

Let $E_{0}=\{e_{g(\alpha)} : \alpha<\omega_{2}\}$. We claim that this set works.

To prove this, it is enough to show that given an uncountable collection of finite open-homogeneous sets, the union of two of them is open-homogeneous.

Assume that we have $\aleph_{1}$-many finite open-homogeneous sets, say $A=\{x_{\alpha} : \alpha<\omega_{1}\}$. We know that there is an uncountable $\beta<\omega_{2}$ such that for all $\alpha<\omega_1$, if $e_{\xi}\in x_{\alpha}$, then $\xi<\beta$. Furthermore, we can find this $\beta$ such that it is closed under $g$.

Without lost of generality, we can assume that they form a $\Delta$-system with empty root and they all have size $n$. As in Theorem \ref{SOCA from SOCA}, we can naturally associate a vector in $E^{n}$ to each $x_{\alpha}$. To do so, we enumerate $x_{\alpha}=\{e_{g(\xi^{1}_{\alpha})}, ..., e_{g(\xi^{n}_{\alpha})}\}$ in such a way that $i<j$ if and only if $b_{\beta}(g(\xi^{i}_{\alpha}))<b_{\beta}(g(\xi^{j}_{\beta}))$.

Shrinking $A$ if necessary, we can assume that, if $\delta<\gamma$, then \[\max(b_{\beta}(\xi^{n}_{\delta}),b_{\beta}(g(\xi^{n}_{\delta})))<\min( b_{\beta}(\xi^{1}_{\gamma}), b_{\beta}(g(\xi^{1}_{\gamma})).\]

After these reductions, let $F$ be the closure of $\{(e_{g(\xi^{1}_{\alpha})}, ..., e_{g(\xi^{n}_{\alpha})}): \alpha<\omega_{1}\}$ in $E^{n}$.

We will work with the continuous $\in$-chain of models associated with $F$ and $\beta$.

Given $\alpha<\beta$, we say that the height of $e_{\alpha}$, denoted $ht^{F,\beta}(e_{\alpha})=ht(e_{\alpha})$, is the minimum $\gamma$ such that $e_{\alpha}\in M_{\gamma+1}\setminus M_{\gamma}$. Given our definition for these chains of models, we have that $ht(e_{\gamma})\leq b_{\beta}(\gamma)$ for $\gamma\in \beta$. On the other hand, the definition of $f_{F,\beta}$ ensure that, if $\alpha, \delta \in \omega_1$ and $f_{F,\beta}(\alpha)<\delta$ then $ht(e_{b^{-1}_{\beta}(\xi)})< \delta $  for all $\xi<\alpha$, $ht(e_{b^{-1}_{\beta}(\alpha)})< \delta $ and $\alpha< ht(e_{b_{\beta}^{-1}(\chi)})$ for all $\chi\geq\delta$.

We claim that, for all but countably many ordinals in $\beta$, if $b_{\beta}(g(\eta))<b_{\beta}(g(\xi))$ then $ht(e_{g(\eta)})<ht(e_{g(\xi)})$. To see this, from the second bullet of \textbf{MTA}, we have that there is a countable $D$ such that if $\eta\neq \xi\in \beta\setminus D$ and $b_{\beta}(g(\eta))<b_{\beta}(g(\xi))$ then $f_{F,\beta}(b_{\beta}(g(\eta))<b_{\beta}(g(\xi))$. This means that $b_{\beta}(g(\eta))<ht(e_{g(\xi)})$. Since $ht(e_{g(\eta)})\leq b_{\beta}(g(\eta))$, we have that $ht(e_{g(\eta)})<ht(e_{g(\xi)})$.

The above inequality allow us to find a really useful family of uncountable sets: given $\alpha<\omega_1$ such that $\xi_{\alpha}^{i}\in \beta\setminus D$ for all $i\in \{1,..., n\}$, let $\mu=ht(e_{g(\xi^{n}_{\alpha})})$ and $F_{\alpha}=\{a\in E\cap M_{\mu}: x_{\alpha}\restrict_{1, ..., n-1}\concat a\in F\}$. Since $b_{\beta}(g(\xi^i_{\alpha}))< b(g(\xi^{n}_{\alpha}))$ for $i<n$, we have that $ht(e_{g(\xi^i_{\alpha})})< ht(e_{g(\xi^n_{\alpha})})=\mu$. Then, $e_{g(\xi^i_{\alpha})}\in M_{\mu}$. Also, $F, x_{\alpha}\restrict_{1, ..., n-1}\in M_{\mu}$, which implies that $F_{\alpha}\in M_{\mu}$. This shows that $F_{\alpha}$ is uncountable. To see this, remember that given a countable set $L$, if $L\in M_{\mu}$ then $L\subseteq M_{\mu}$. Therefore, the fact that $F_{\alpha}\in M_{\mu}$, $\xi^n_{\alpha}\in F_{\alpha}$ but $\xi^n_{\alpha}\notin M_{\mu} $ shows that $F_{\alpha}$ is uncountable.

From here, it is enough to follow the exposition of the consistency of \textbf{SOCA($\aleph_1$)} as in \cite{kunen2014set} Lemmas V.6.14 and V.6.15. The technique presented there is the club method used in \cite{abraham1985consistency}. The club method has two steps: the preparation and the cloning. For the preparation, we select open sets for each $\alpha<\omega$ as we exposed here for Theorem \ref{SOCA from SOCA}. 

Cloning resembles the technique in Theorem \ref{Luzin entangled} where we use the oscillation of a discontinuous function, although the contexts are really different. The analogy comes from the fact that both techniques first find points in a set of accumulation points ($F$ and $F_{\alpha}$ here and the oscillation in Theorem \ref{Luzin entangled}) to fix open sets ($I_s$, $I_r$, $(-\infty,r)$ and $(r,\infty)$ in Theorem \ref{Luzin entangled}). Afterwards, both technique uses those open sets to find actual elements of the original uncountable set that are compatible (elements of $\{(e_{g(\xi^{1}_{\alpha})}, ..., e_{g(\xi^{n}_{\alpha})}): \alpha<\omega_{1}\}$ here and elements of $K$ in Theorem \ref{Luzin entangled}).
\end{proof}

We see these results as steps towards the consistency of \textbf{SOCA($\aleph_2$)}.

\begin{conjecture}
\emph{\textbf{SOCA($\aleph_2$)}} is consistent.
\end{conjecture}

Proving the conjecture would culminate the work that Ken and I talked so much about.

\section{Once you see the stars}

Once you see the stars, is impossible to forget them. I was Kenneth Kunen's last student. He accepted to be my advisor even though he was already retired.

Being Ken's student forced me to face freedom and taught me how to find questions on my own. I learned early on that Ken was not going to tell me which path to follow, but he was also not going to let me go astray if I needed him.

In 2020, Ken fell ill but, being as responsible and caring as he was, he attended  the defence of my disertation \cite{ongay2020relations} virtually. According to his family, a couple of days later he was hospitalized. He died on August 14, 2020, the same day as my birthday.

\section{Acknowledgements}

The author wants to kindly thank the editors of this issue, Juris Stepr\={a}ns and the referee for all the suggestions and comments. They helped to improve this article to its current form. I also thank my son and wife for all the support they gave me during my PhD and the hard time that came after it.

\newpage

\bibliographystyle{abbrv}
\bibliography{biblio}

\end{document}